\documentclass[a4paper]{amsart}
  \usepackage{amsmath,amsthm,amssymb,exscale,bbm,xpatch,mathrsfs,nicefrac,hyperref}
  \usepackage[utf8]{inputenc}

\calclayout

 \newtheorem{thm}{Theorem}[section]
 
 \newtheorem{lemma}[thm]{Lemma}
 
 \theoremstyle{definition}
 
 \newtheorem{remark}[thm]{Remark}
 \newtheorem{remarks}[thm]{Remarks}

 \numberwithin{equation}{section}

\newcommand{\Om}{\Omega}
\newcommand{\Borel}{\mathrm{Borel}}

\newcounter{mattner}
\newenvironment{mattnerlist}[1]{\begin{list}{ {\upshape[{#1}\arabic{mattner}]}}{
        \usecounter{mattner}
        \topsep1ex
        \parsep0cm
        \itemsep0.5ex
        \leftmargin1.5cm
        \labelwidth1cm
        \labelsep0.5cm
        \itemindent0cm
}}
{\end{list}}

\newcommand{\BM}{\mathrm{BM}}

\newcommand{\vanish}[1]{\relax}
\newcommand{\beq}{\begin{equation}}
\newcommand{\eeq}{\end{equation}}

\newcommand{\Ball}{\mathrm{Ball}}

\newcommand{\hs}{\hskip-0.1em}
\newcommand{\set}[1]{\hs\left[\,#1\,\right]}

\newcommand{\prfnoi}{\smallskip\noindent}
\newcommand{\dps}{\displaystyle}

\newcommand{\Ge}{\mathrm{G}}

\newcommand{\gebiet}{O}

\newcommand{\upi}{\pi}

\newcommand{\bignorm}[1]{\bigl\| #1 \bigr\|}

\newcommand{\emdf}{\bf}

\newcommand{\suchthat}{\,\,|\,\,}

\DeclareMathOperator{\wessup}{\text{\upshape ess.sup}}

\newcommand{\N}{\mathbb{N}}

\newcommand{\R}{\mathbb{R}}
\newcommand{\C}{\mathbb{C}}

\newcommand{\ud}{\mathrm{d}}

\newcommand{\ui}{\mathrm{i}}

\newcommand{\Ha}{\mathrm{H}}

\newcommand{\vphi}{\varphi}

\newcommand{\Sum}[2][\relax]{%
 \ifx#1\relax \sideset{}{_{#2}}\sum 
 \else \sideset{}{^{#1}_{#2}}\sum
 \fi}

\newcommand{\car}{\mathbf{1}}

\newcommand{\abs}[1]{\vert #1 \vert}

\newcommand{\Ell}[1]{\mathrm{L}_{#1}}

\newcommand{\dann}{\Rightarrow}

\newcommand{\gdw}{\Leftrightarrow}

\newcommand{\nach}{\circ}

\newcommand{\cls}[1]{\overline{#1}}

\newcommand{\spann}{\mathrm{span}}

\newcommand{\tensor}{\otimes}

\newcommand{\BL}{\mathcal{L}}

\newcommand{\norm}[2][\relax]{%
   \ifx#1\relax \ensuremath{\lVert#2\rVert}
   \else \ensuremath{\left\Vert#2\right\Vert_{#1}}
   \fi}

\makeatletter
\newcommand{\sprod}[2]{\ensuremath{%
  \setbox0=\hbox{\ensuremath{#2}}
  \dimen@\ht0
  \advance\dimen@ by \dp0
  \left[ #1\rule[-\dp0]{0pt}{\dimen@}, #2\hspace{1pt}\right]}}
\newcommand{\bsprod}[2]{\ensuremath{%
  \setbox0=\hbox{\ensuremath{#2}}
  \dimen@\ht0
  \advance\dimen@ by \dp0
  \bigl[ #1\rule[-\dp0]{0pt}{\dimen@}, #2\hspace{1pt}\bigr]}}

 \makeatother
\newcommand{\dprod}[2]{\ensuremath{\langle#1,#2\rangle}}

\newcounter{aufzi}
\newenvironment{aufzi}{\begin{list}{ {\upshape\alph{aufzi})}}{
        \usecounter{aufzi}
        \topsep1ex
        \parsep0cm
        \itemsep1ex
        \leftmargin0.8cm
        \labelwidth0.5cm
        \labelsep0.3cm
}}
{\end{list}}

\newcounter{aufzii}
\newenvironment{aufzii}{\begin{list}{\hfill {\upshape 
(\roman{aufzii})}}{
        \usecounter{aufzii}
        \topsep1ex
        \parsep0cm
        \itemsep1ex
        \leftmargin0.8cm
        \labelwidth0.5cm
        \labelsep0.3cm
         \itemindent0cm
}}
{\end{list}}

\newcounter{aufziii}

\begin{document}

\title[Vector-Valued Holomorphic Functions and
Fubini-Type Theorems]{Vector-Valued Holomorphic Functions and
Abstract Fubini-Type Theorems}

\author[Bernhard H. Haak]{Bernhard H. Haak}
\address{%
  Institut de Math\'ematiques de Bordeaux\\
  Universit\'e de Bordeaux\\
  351, cours de la Lib\'eration\\
  33405 Talence cedex\\
  France}
\email{bernhard.haak@math.u-bordeaux.fr}
\author[Markus Haase]{Markus Haase}
\address{%
Mathematisches Seminar der CAU Kiel \\
Heinrich-Hecht-Platz. 6, 24118 Kiel\\
Germany\\
}
\email{haase@math.uni-kiel.de}

\subjclass{%
  46G20,
  46E15,
  32A10.
}
\keywords{vector-valued, holomorphic functions of several variables,
  linearization theorem}

\date{\today}

\begin{abstract}
Let  $f = f(z,t)$ be a function holomorphic in 
$z \in O \subseteq \C^d$ for fixed $t\in \Om$ 
and measurable in $t$ for fixed $z$ and such that
$z \mapsto f(z,\cdot)$ is bounded with values in
$E := \Ell{p}(\Om)$, $1\le p \le \infty$. It is proved (among other things) that 
\[ \dprod{t\mapsto \vphi( f(\cdot,t))}{\mu}
= \vphi(z \mapsto  \dprod{f(z, \cdot)}{\mu})
\] 
whenever $\mu \in E'$ and $\vphi$ is a bp-continuous
linear functional on $\Ha^\infty(O)$. 
\end{abstract}

\maketitle

\allowdisplaybreaks

\section{Introduction}\label{intro}

Let $\gebiet$ and $\Om$ be non-empty sets and let
\[ f: \gebiet \times \Omega \to \C,\qquad (z,t) \mapsto  f(z, t)
\]
be a function from which we define two functions,
\[
  F(z) := f(z,\cdot) \in \C^\Om\qquad (z\in \gebiet)\qquad \text{and} \qquad
f_t := f(\cdot, t) \in \C^\gebiet\qquad  (t\in \Om).
\]
Given linear functionals $\mu$ on $\C^\Om$ and $\vphi$ on
$\C^\gebiet$, we might ask whether their application commutes,
i.e. whether one has \beq\label{intro.eq.Fubini} \dprod{t\mapsto
  \vphi( f_t)}{\mu} = \vphi(z \mapsto \dprod{F(z)}{\mu}).  \eeq We
shall call a theorem stating the validity of
\eqref{intro.eq.Fubini} under certain conditions an {\em abstract Fubini-type theorem}, for
obvious reasons.

\smallskip

$\Ell{2}$-valued bounded holomorphic functions play a prominent role
in our work on functional calculus
\cite{HaakHaase:sqfc-arxiv,HaakHaase:sqfc-Asterisque}, and we came
naturally across some particular instances of abstract Fubini-type
theorems in that context.  It is the purpose of this note to present
these results independently of their relevance for functional calculus
theory, because in our view they are interesting in their own right.

\smallskip

The abstract Fubini-type theorems we are aiming at involve holomorphy
in the first and measurability in the second variable.  To wit, we suppose
that $\gebiet \subseteq \C^d$ is open, $\Omega$ is a measure space and
$f: \gebiet \times \Om \to \C$ is such that $f_t$ is holomorphic for
each $t\in \Om$ and $F(z)$ is measurable for each $z\in \gebiet$ (plus
additional hypotheses). Thus, contrary to the classical Fubini
theorem, there is a built-in asymmetry motivated by the aim to
eventually regard $F$ as a bounded holomorphic function with values in
some Banach space $E$ of (equivalence classes of) measurable
functions.

\medskip

The paper is organized as follows. In Section~\ref{s.hol} we recall
classical results for vector-valued holomorphic functions. Next, in
Section~\ref{s.lin}, we state a corollary of Mujica's linearization
theorem from
\cite{Mujica1991}, but with a new proof. In Section~\ref{s.bdd}, we
consider first the case of a closed subspace $E$ of the space $\BM(\Om,
\Sigma)$ of bounded measurable
functions. The main result is then 
Theorem~\ref{int.t.ifHinf}.  It
treats the case $E= \Ell{p}(\Om)$ for $1\le p
< \infty$. It is complemented by Theorem~\ref{int.t.Linf}, which deals
with $E = \Ell{\infty}(\Om)$.
 
Theorem~\ref{int.t.ifHinf} extends
results of Mattner \cite{Mattner2001}, who only treats the case
$d{=}1$ and $ p{=}1$.  In the concluding Section~\ref{s.matt} we
discuss Mattner's theorem and its relation to our work. In particular,
we present an alternative proof of Theorem~\ref{int.t.ifHinf} for the
case $p{=}1$ using Mattner's results as a starting point.

\smallskip 
As it is often the case with papers in functional analysis, the thrust
of the presentation is not just about results (which are only
partially new) but also about proofs. In particular, we strive to keep
the presentation self-contained, at least for readers with a
background in functional analysis.

\subsubsection*{Terminology  and Notation}
Generically, $E, F$ denote complex Banach spaces, and $E', F'$ are their
respective duals; the canonical duality 
$E \times E'\to \C$ is denoted by
$(x, x') \mapsto \dprod{x}{x'}$.

\prfnoi
For any set $\Om$ we write $\ell^\infty(\Om; E)$
for the space of all bounded $E$-valued functions on $\Om$, endowed
with the supremum norm; moreover, we abbreviate $\ell^\infty(\Om) := \ell^\infty(\Om;\C)$.

We say that a sequence $(f_n)_n$ in $\ell^\infty(\Om; E)$
{\em bp-converges} to $f: \Om \to E$ if $f_n \to f$ pointwise on $\Om$
and $\sup_n \norm{f_n}_\infty < \infty$.

A subset $M \subseteq \ell^\infty(\Om;E)$ is called {\em
  bp-closed} if it is closed under taking bp-limits of
sequences in $M$.

Let $F \subseteq \ell^\infty(\Om;E)$ be a bp-closed
subspace.  A linear mapping $\vphi: F \to X$ (where $X$ is any
topological vector space) is called {\em bp-continuous} if
\[  f_n \to f \quad\text{(bp)}\qquad
\dann\quad T f_n \to T f.
\]
We agree that it would be more accurate to speak of ``sequentially
bp-closed''
sets and ``sequentially bp-continuous'' mappings. However, we decided to
drop the word ``sequentially'' for the sake of brevity. 

\prfnoi For any measurable space $(\Om,\Sigma)$ we let
$\BM(\Om) := \BM(\Om,\Sigma)$ be the Banach space of all bounded and
measurable $\C$-valued functions, endowed with the supremum norm. This
is a bp-closed subspace of $\ell^\infty(\Om)$.

\section{Vector-Valued Multivariate Holomorphic Functions}\label{s.hol}

Here and in the following, $\gebiet \subseteq \C^d$ is a fixed open
and not-empty set and $E$ is a complex Banach space.  A function
$f: \gebiet \to E$ is {\emdf holomorphic} if it is totally
differentiable (= Fréchet-differentiable) with $\C$-linear derivative
at each point of $\gebiet$ \cite[Def.~147,~p.68]{HajekJohanisSABS}.
The following theorem is a useful characterization of holomorphy.  It
extends well-known results for multi-variable scalar-valued and
one-variable vector-valued functions.

\begin{thm}\label{hol.t.char}
  Let $E$ be a complex Banach space, and $N \subseteq E'$ an
  $E$-norming subset of $E'$.  For a mapping $f: \gebiet \to E$ the
  following assertions are equivalent:
\begin{aufzii}
\item \label{i.hol.t.char.a} $f$ is holomorphic.
\item \label{i.hol.t.char.b} $f$ is continuous and separately
  holomorphic.
\item \label{i.hol.t.char.c} $f$ is locally bounded and $x'\nach f$ is
  separately holomorphic for all $x'\in N$.
\end{aufzii}
In this case for each $a= (a_j)_{j=1}^d \in O$ and each $r > 0$ with
$\prod_{j=1}^d \Ball[a_j;r] \subseteq O$ one has the {\emdf Cauchy
  formula}
\[
  f(z) = \frac{1}{(2\upi\ui)^d} \int_{\abs{w_d-a_d} = r} \dots
  \int_{\abs{w_1-a_d}=r} \frac{f(w)\, \ud{w_1} \dots \ud{w_d}}{(w_1-
    z_1)\cdots (w_d -z_d)}
\]
for all $z\in \C^d$ with $\abs{z-a}_\infty < r$.
\end{thm}

\noindent Here, ``separately holomorphic'' means partially complex
differentiable in each coordinate direction.  For the proof of Theorem
\ref{hol.t.char} we recall the following simple consequence of
Schwarz' lemma.

\begin{lemma}\label{hol.l.Schwarz}
  Let $a\in \C, r > 0$ and $f: \Ball(a;r) \to \C$ holomorphic and
  bounded. Then
  \[
    \abs{f(z) - f(a)} \le \frac{2}{r} \norm{f}_\infty \abs{z-a} \qquad
    (\abs{z-a} < r).
  \]
\end{lemma}

\begin{proof}
  Fix $M > \norm{f}_\infty$ and define $\vphi : \Ball(0;1) \to \C$ by
  $\vphi(z) := \frac{1}{2M} (f(a + rz) - f(a))$. Then $\vphi$ is
  holomorphic and bounded by $1$ on $\Ball(0;1)$ with $\vphi(0)=
  0$. By Schwarz' Lemma, $\abs{\vphi(z)} \le \abs{z}$ for
  $\abs{z} < 1$. Replacing $z$ by $\frac{1}{r}(z-a)$ yields the claim.
\end{proof}

\begin{proof}[Proof of Theorem~\ref{hol.t.char}]
  Clearly, \ref{i.hol.t.char.a} implies \ref{i.hol.t.char.b} and
  \ref{i.hol.t.char.b} implies \ref{i.hol.t.char.c}. Now suppose
  \ref{i.hol.t.char.c} and let $D := \prod_j \Ball(a_j;r)$ be any
  polydisc with $\cls{D} \subseteq \gebiet$.  Note that, by
  hypothesis, $f$ is bounded on $\cls{D}$.

  We first show as in \cite[proof~of~1.3]{KaupKaupHFSCV}
  that $f$ is continuous at $a$. Let $z = (z_j)_j \in D$ and write
  \[
    f(z) - f(a) = \sum_{j=1}^d f(z_1, \dots, z_{j}, a_{j{+}1}, \dots,
    a_d) - f(z_1, \dots, z_{j{-}1}, a_{j}, \dots, a_d).
  \]
  Composing with $x'\in N$ and applying Lemma~\ref{hol.l.Schwarz} to
  each summand we obtain
  \[
    \abs{x'(f(z) - f(a))} \le \frac{2}{r} \norm{f}_\infty \sum_{j=1}^d
    \abs{z_j -a_j}.
  \]
  Taking the supremum over $x'\in N$ yields
  \[
    \norm{f(z) - f(a)} \le \frac{2}{r} \norm{f}_\infty \sum_{j=1}^d
    \abs{z_j -a_j}\quad \text{whenever}\quad \abs{z-a}_\infty < r.
  \]
  In particular, $f$ is continuous at $a$. Since $a\in \gebiet$ was
  arbitrary, $f$ is continuous.

  Continuity of $f$ implies that the function
  \[
    g(z) := \frac{1}{(2\upi\ui)^d} \int_{\abs{w_d-a_d} = r} \dots
    \int_{\abs{w_1-a_d}=r} \frac{f(w)\, \ud{w_1} \dots \ud{w_d}}{(w_1-
      z_1)\cdots (w_d -z_d)} \qquad (z\in D)
  \]
is a well-defined $E$-valued holomorphic function on $D$.  Indeed, $g$
is certainly holomorphic in each variable separately and each partial
derivative is continuous. Hence one can apply
\cite[XIII,~Thm.~7.1]{LangRFA}.

Composing with $x'\in N$ yields, by the scalar Cauchy
formula in each variable separately, the identity
\[
  x'(g(z)) = x'(f(z)) \quad (z\in D,\,x'\in N).
\]
Since $N$ is norming (in particular: separating), it follows that
$g=f$ on $D$. Hence, $f$ is holomorphic on $D$.  This implies
\ref{i.hol.t.char.a}.
\end{proof}

\begin{remarks}
  \prfnoi 1)\ For univariate functions, Theorem~\ref{hol.t.char} is
  well-known, see \cite[Appendix~A]{ABHN}.  For scalar-valued
  functions, the equivalence \ref{i.hol.t.char.a} $\gdw$
  \ref{i.hol.t.char.b} is called Osgood's lemma; the (a priori)
  stronger equivalence \ref{i.hol.t.char.a} $\gdw$
  \ref{i.hol.t.char.c} is \cite[Thm.~1.3]{KaupKaupHFSCV}.

  \prfnoi 2)\ Many books start from a different definition of
  holomorphy than ours and do not even mention
  Fréchet-differentiability. A noteworthy exception is
  \cite{HajekJohanisSABS}. See also
  \cite[Theorem~160]{HajekJohanisSABS} for further characterizations
  of holomorphy.

  \prfnoi 3)\ Assertion \ref{i.hol.t.char.c} involving a norming
  subset is due to Grothendieck \cite{Grothendieck1953}, see also
  \cite[p.139]{KatoPTLO}.  It implies (via the Hahn--Banach theorem)
  the following equivalences:
  \begin{aufzi}
  \item A function $f: O \to E$ is holomorphic if and only if it is
    weakly holomorphic (Dunford's theorem, see
    \cite[Thm.~148,~p.68]{HajekJohanisSABS}).
  \item A function $T: O \to \BL(E;F)$ is holomorphic if and only if
    it is strongly holomorphic, if and only if for each $x\in E$ and
    each $x'\in N$ from a norming subset $N \subseteq F'$ the function
    $\dprod{F(\cdot)x}{x'}$ is holomorphic.
  \end{aufzi}

  \prfnoi 4) Theorem~\ref{hol.t.char} still holds if $N$ is merely a
  {\em separating} subset of $E'$. This follows from the analogous
  result for univariate functions, due to Grosse-Erdmann
  \cite{Grosse-Erdmann1992}. Different proofs have been given by
  Arendt and Nikolski \cite[Thm.~3.1]{ArendtNikolski2000} (see also
  \cite{Arendt2016}) and Grosse-Erdmann
  \cite[Thm.~1]{Grosse-Erdmann2004}. For more general results in this
  direction see Frerick, Jord\'{a} and Wengeroth
  \cite{FrerickJordaWengenroth2009}, in particular their Theorem~3.2.

  \prfnoi 5)\ One may replace \ref{i.hol.t.char.b} by the weaker
  assertion
  \begin{aufzi}
  \item[ (ii)'] $f$ is weakly separately holomorphic.
  \end{aufzi}
  This follows from Hartogs' theorem, which says that a scalar
  multi-variable function is already holomorphic if it is merely
  separately holomorphic \cite[Thm.~153,~p.69]{HajekJohanisSABS}.
\end{remarks}

\noindent It follows easily from the Cauchy integral formula that each
partial derivative $\frac{\partial}{\partial z_j} f$ of a holomorphic
function $f: \gebiet \to E$ is again holomorphic.  Iterating this
yields for $\alpha \in \N_0^d$ the holomorphic function
\[  D^\alpha f := \prod_{j=1}^d \frac{\partial^{\alpha_j}}{\partial
    z_j^{\alpha_j}} f,
\]
and one has the well-known Cauchy integral formula for derivatives. 
From there, the following lemma is straightforward.

\begin{lemma}\label{hol.l.lin}
Let $f: O \to E$ holomorphic and $T: E \to F$ bounded and linear.
Then $T \nach f$ is  holomorphic and  $D^\alpha(T \nach f) = T \nach
D^\alpha f$. 
\end{lemma}

\noindent In the case that $f= f(z,t) : \gebiet \times \Om \to \C$
depends on an additional parameter $t\in\Om$, we shall write
\[    D^\alpha_z f(a,t) \mapsto D^\alpha(f(\cdot, t))(a) \qquad
(a,t)\in \gebiet\times \Om.
\]
We write $\Ha^{\infty}(\gebiet; E)$ for the space of all bounded
holomorphic $E$-valued functions on $\gebiet$.
If we endow $\gebiet$ with the Borel $\sigma$-algebra,
$\Ha^{\infty}(\gebiet;E)$ becomes a (sequentially) bp-closed subspace
of $\BM(\gebiet;E)$. From the Cauchy integral formula it follows that
for fixed $a\in \gebiet$ and $\alpha \in \N_0^d$ the mapping
\[  \Ha^\infty(\gebiet;E) \to E,\qquad f \mapsto  D^\alpha f(a)
\]
is bp-continuous.

\section{The Linearization Theorem}\label{s.lin}

The following theorem is a corollary of Mujica's linearization theorem
\cite[Thm.~2.1]{Mujica1991}, however with a different proof (see
Remarks \ref{hol.r.linearization} below).

\begin{thm}[Linearization]\label{hol.t.linearization}
Let $\gebiet\subseteq \C^d$ be open, $E$ a Banach space and $f\in
\Ha^\infty(\gebiet; E)$. Then for each bp-continuous functional 
$\vphi \in \Ha^\infty(\gebiet)'$ there is a unique element $\vphi_f \in
E$ such that 
\[     \dprod{\vphi_f}{x'} = \dprod{x'\nach f}{\vphi}\qquad (x'\in
E').
\]
Moreover, $\vphi_f \in \cls{\spann}f(\gebiet)$.
\end{thm}

\begin{proof}
Uniqueness and the second assertion follow from the Hahn--Banach
theorem.  For existence we may suppose without loss of generality that 
$E =\cls{\spann}f(\gebiet)$. As $f$ is continuous and $\gebiet$ is
separable, so is $E$.   Define the
  linear functional $\vphi_f: E'\to \C$ by
  \[
    \vphi_f(x') := \vphi( z \mapsto (x'\nach f)(z)) \qquad ( x' \in E').
\]
We need to prove that  $\vphi_f\in E$ under the natural embedding
$E\hookrightarrow E''$. Since $E$ is
separable, the weak$^*$-topology on the unit ball of $E'$ is
metrizable.  Since $\vphi$ is bp-continuous, $\vphi_f$ is weakly$^*$
sequentially continuous on the unit ball of $E'$, and hence weakly$^*$
continuous there. Then, by a  well-known theorem of Banach (see
\cite[Lemma~1.2]{Stratila-Zsido} for an elegant proof), it follows
that $\vphi_f\in E$ as claimed.
\end{proof}

\begin{remarks}\label{hol.r.linearization}
1) As it stands, Theorem~\ref{hol.t.linearization} is 
a consequence of Mujica's linearization theorem \cite[Theorem
2.1]{Mujica1991}. There, Theorem~\ref{hol.t.linearization} is
stated for open $\gebiet \subseteq F$, where $F$ is any 
Banach space, and $\vphi \in \Ge^\infty(\gebiet)$. Here,
$\Ge^\infty(\gebiet)$ is the space of all linear
functionals which on bounded
subsets of $\Ha^\infty(\gebiet)$ are
continuous with respect to $\tau_c$,  
the topology of uniform convergence on compacts. 

Now, since any open set $\gebiet \subseteq \C^d$
is locally compact and $\sigma$-compact, the topology $\tau_c$ is metrizable, and hence
on bounded subsets of $\Ha^\infty(\gebiet)$,
bp-continuity and $\tau_c$-continuity of a functional coincide. 

\prfnoi 2) Theorem~\ref{hol.t.linearization} remains true (with essentially the same proof) under the
following weaker hypotheses: $F$ is any
Banach space,  $\gebiet \subseteq F$ open, and $f\in
\Ha^\infty(\gebiet; E)$ is separably-valued.

Under these hypotheses, a bp-continuous functional may not be 
$\tau_c$-continuous on bounded sets, and hence the result may then not be
covered by Mujica's theorem. (Unfortunately, we do not know of an
example showing that this is indeed the case.)

\prfnoi 3)\ Our proof of Theorem~\ref{hol.t.linearization}  is more
elementary than Mujica's from \cite{Mujica1991}, as it avoids the
bipolar theorem and is built just on Banach's theorem (which has a
much more elementary proof). 
 Moreover, a slight modification of our argument 
also works in the more general setting of Mujica's theorem: Namely, if 
$(x_\alpha')_\alpha$ is a norm-bounded net and weakly$^*$-convergent to
$x'\in E'$, then  the net $(x_\alpha'\nach f)_\alpha$ is
$\tau_c$-convergent to $x'\nach f$, by equicontinuity.

Using this argument leads to a simplification of the proof of Mujica's theorem.
However,  it should be noted that the linearization result 
is only one (although essential) part of Mujica's
original theorem from \cite{Mujica1991}. 
\end{remarks}

The proof of Theorem
\ref{hol.t.linearization} is primarily functional-analytic.
It may be interesting to see that one can alternatively employ
concepts from measure theory. Here, we suppose
in addition that $\vphi$ is integration with respect to a
measure. (Cf., however, Remark~\ref{rem:cooper} below.)

\begin{proof}[Second Proof of Theorem~\ref{hol.t.linearization}]
  Let the functional $\vphi: \Ha^\infty(O) \to \C$ be given by
  integration against a complex Borel measure $\nu$. 
Since $f$ is holomorphic, it is 
  weakly holomorphic and has values in a separable subspace.  Hence,
  by Pettis' measurability theorem, $f$ is strongly
  $\abs{\nu}$-measurable. Since $f$ is bounded, it is Bochner
  integrable with respect to $\abs{\nu}$, and hence the Bochner integral
  \[
    h := \int_\gebiet f \, \ud{\nu} \in E
  \]
exists. Applying $x' \in E'$ yields
  \[
    \dprod{h}{x'} = \int_{\gebiet} \dprod{f(\cdot)}{x'} \, \ud{\nu} =
    \vphi( z \mapsto \dprod{f(z)}{x'}).\qedhere
  \]
\end{proof}

\begin{remark}\label{rem:cooper}
  The assumption that the bp-continuous functional $\vphi\in \Ha^\infty(\gebiet)'$ is integration with respect to a measure
  is only virtually restrictive. Indeed, it turns out that each
  bp-continuous functional on $\Ha^\infty(\gebiet)$ is already given
  by integration against some finite measure, and even one that has a
  density with respect to Lebesgue measure. This follows from
  identifying the space of bp-continuous functionals on
  $\Ha^\infty(\gebiet)$ with the dual of $\Ha^\infty(\gebiet)$ with
  respect to the so-called {\em mixed topology}, see
  \cite[Chap.~V.1]{CooperSSAATFA}, in particular part~5) of
  Proposition~1.1 and Proposition~1.2. (Actually, Cooper \cite
  {CooperSSAATFA} only treats the case $d=1$, but we expect the
  same result for $d >1$.)
\end{remark}

\section{Holomorphic Families of Measurable and Integrable
  Functions}\label{s.bdd}\label{s.int}

The Linearization Theorem~\ref{hol.t.linearization} takes the form of an abstract
Fubini-type theorem if $E$ is a space of functions on a set $\Omega$.
We exploit this idea first for Banach spaces of bounded functions.
Note that every space $E= \BM(\Om, \Sigma)$ for some $\sigma$-algebra $\Sigma$ on a set $\Om$ is a closed subspace of
$\ell^\infty(\Om)$.

\begin{thm}\label{bdd.t.mfHinf}
  Let $\Om$ be a set and $E$ a closed subspace of $\ell^\infty(\Om)$,
  let $\gebiet \subseteq \C^d$ be open and
  $f : \gebiet \times \Omega \to \C$ a bounded function with the
  following properties:
\begin{itemize}
\item  $F(z) = f(z, \cdot) \in E$ for each $z\in \gebiet$ and
\item  $f_t = f(\cdot, t)$ is holomorphic for each $t\in \Omega$.
\end{itemize}
Finally, let  $\vphi: \Ha^\infty(\gebiet) \to \C$ be a bp-continuous
functional. Then the following
assertions hold:
\begin{aufzi}
\item \label{i.bdd.t.mfHinf.a} $F \in \Ha^\infty(\gebiet; E)$.

\item \label{i.bdd.t.mfHinf.b} For each multi-index $\alpha \in
\N_0^d$ and each $a\in \gebiet$ one has
\[   (D^\alpha F)(a) =   t \mapsto D_z^\alpha f(z,t)\Big|_{z=a}  \in E
\]
and, for all $\mu \in E'$,\quad 
$\displaystyle    D^\alpha_z \dprod{ t\mapsto f(z,t)}{\mu}\Big|_{z=a} =     \dprod{t\mapsto
  D_z^\alpha f(z,t)\Big|_{z=a} }{\mu}$.

\item \label{i.bdd.t.mfHinf.c} The function
$t\mapsto \vphi( f_t)$ is  contained in 
$\cls{\spann}( F(\gebiet)) \subseteq E$.

\item \label{i.bdd.t.mfHinf.d} For all $\mu \in E'$: \quad 
$\dprod{t\mapsto \vphi( f_t)}{\mu}
= \vphi(z \mapsto  \dprod{F(z)}{\mu})$.
\end{aufzi} 
\end{thm}

\noindent Note that the second part of \ref{i.bdd.t.mfHinf.b} is of the type ``differentiation under
the integral'' when one is inclined to interpret the application of
$\mu$ as some kind of integral.

\begin{proof}[Proof of Theorem~\ref{bdd.t.mfHinf}]
  Observe that the set of Dirac (= point evaluation) functionals
  $\{ \delta_t \suchthat t\in \Omega\}$ is norming for $E$. Since $F$
  is bounded and (by hypothesis) the functions
  \[
    z \mapsto \dprod{F(z)}{\delta_t} = f(z,t) \qquad(t\in \Omega)
\]
are all holomorphic, $F$ is holomorphic by Theorem~\ref{hol.t.char}.
It follows that $D^\alpha F$ takes values in $E$ and
\[
  \dprod{(D^\alpha F)(z)}{\mu}  = D_z^\alpha \dprod{F(z)}{\mu}
\]
as scalar functions on $\gebiet$ for each $\mu \in E'$ (Lemma
\ref{hol.l.lin}).  Specializing $\mu = \delta_t$ for $t\in \Om$ yields
$(D^\alpha F)(z)(t) = D^\alpha_z f(z, t)$, and hence
\ref{i.bdd.t.mfHinf.b} is proved.

Next, we apply Theorem~\ref{hol.t.linearization} to the mapping $F\in
\Ha^\infty(\gebiet; E)$ and obtain, for any bp-continuous functional
$\vphi \in \Ha^\infty(\gebiet)'$ an element $\vphi_F\in 
\cls{\spann}(F(\gebiet))$ with 
\[   \dprod{\vphi_F}{\mu} = \vphi( z \mapsto \dprod{F(z)}{\mu}) 
\qquad (\mu \in E').
\]
By specializing $\mu = \delta_t$ for $t\in \Omega$, we find
$\vphi_F = (t \mapsto \vphi(f_t))$. This concludes the proof of
\ref{i.bdd.t.mfHinf.c} and \ref{i.bdd.t.mfHinf.d}.
\end{proof}

\noindent
Next, we replace the {\em set} $\Omega$ by a {\em
  measure space} $(\Om,\Sigma,\mu)$.  The corresponding
$\Ell{p}$-space is denoted by $\Ell{p}(\Omega)$, and $p,q$ are always
dual exponents, i.e., $1\le p,q \le \infty$ with
$\frac{1}{p} + \frac{1}{q} = 1$.

Before we turn to the main results, let us fix some 
auxiliary information.

\begin{lemma}\label{int.l.meas}
Let $(\Omega, \Sigma)$ be a measurable space,  $\gebiet \subseteq 
\C^d$ an open subset of $\C^d$ and
\[  f: \gebiet \times \Omega \to \C
\]
a function with the following properties:
\begin{itemize}
\item $F(z) := f(z, \cdot)$ is measurable for each $z\in \gebiet$ and \smallskip
\item $f_t := f(\cdot, t)$ is bounded and holomorphic for each $t\in
  \Omega$.\smallskip
\end{itemize}
Then the following assertions hold.
\begin{aufzi}
\item For each $\alpha \in \N_0^d$ and $a\in \gebiet$, the function
$D^\alpha_z f(a, \cdot)$ is measurable.
\item The function $t \mapsto \norm{f_t}_\infty = \sup_{z\in \gebiet}
  \abs{f(z, t)}$ is measurable.

\item For each bp-continuous functional $\vphi : \Ha^\infty(\gebiet)
  \to \C$ the function $t \mapsto \vphi(f_t)$ is measurable.  
\end{aufzi}
\end{lemma}

\begin{proof}
a)\  By induction, it suffices to prove the statement for the case
$\abs{\alpha} = 1$. As a partial derivative is a limit of a sequence of
difference quotients, the claim follows. 

\prfnoi
b)\ Note that, by hypothesis, $f_t \in \Ha^\infty(\gebiet)$, i.e., 
$\norm{f_t}_\infty < \infty$ for each $t\in  \Omega$. Let
$D\subseteq \gebiet$ be a countable dense set. Then, since each $f_t$
is continuous, 
\[  c(t) := \norm{f_t}_\infty = \sup_{z\in \gebiet} \abs{f(z,t)} = 
\sup_{z\in D} \abs{f(z,t)}\qquad (t\in \Om).
\]
Hence $c: \Omega \to \R_+$ is a pointwise supremum of countably many
measurable
functions, and hence measurable.

\prfnoi
c)\  Define
\[   m(z, t) := \frac{f(z,t)}{1 + c(t)} \qquad (z\in \gebiet,\, t\in
\Omega).
\]
Then $\abs{m(z,t)} \le 1$ for all $(z,t)\in \gebiet\times \Omega$. 
Hence, $m$ meets the conditions of Theorem~\ref{bdd.t.mfHinf} with $E= \BM(\Omega,
\Sigma)$,
the space of bounded $\Sigma$-measurable functions on $\Omega$. Given
a bp-continuous functional $\vphi$ on $\Ha^\infty(\gebiet)$ it follows that 
\[ t\mapsto \vphi(m_t) = \frac{\vphi( f_t)}{1 + c(t)}
\]
is measurable. As $1 + c$ is measurable, so is
$t \mapsto \vphi(f_t)$. 
\end{proof}

The next is our main theorem. Observe that assertion
\ref{i.int.t.ifHinf.d} is an abstract Fubini-type result as in
\eqref{intro.eq.Fubini}.

\begin{thm}\label{int.t.ifHinf}
Let $(\Omega, \Sigma, \mu)$ be a measure space, $1\le p < \infty$,  $\gebiet \subseteq 
\C^d$ an open subset of $\C^d$ and
\[  f: \gebiet \times \Omega \to \C
\]
a function with the following properties:
\begin{itemize}
\item $\forall z \in \gebiet: \; F(z) := f(z, \cdot)$ is measurable 
  and   $\displaystyle \sup_{z\in \gebiet} \int_\Omega \abs{ f(z,t)}^p\,
  \mu(\ud{t}) <
  \infty;$\smallskip
\item $\forall t \in \Omega: \; f_t := f(\cdot, t)$ is bounded and holomorphic.\smallskip
\end{itemize}
Fix $\alpha \in \N_0^d$, $a\in \gebiet$,  and a bp-continuous linear functional $\vphi:
\Ha^\infty(\gebiet)\to \C$. 
 Then the following
assertions hold:
\begin{aufzi}
\item \label{i.int.t.ifHinf.a}
  $F \in \Ha^\infty(\gebiet; \Ell{p}(\Omega))$
and  
  $D^\alpha F(a) = D^\alpha_z f(a,\cdot)$ $\mu$-almost everywhere.
\item \label{i.int.t.ifHinf.b} For each $h\in \Ell{q}(\Om)$ the
  function $\dps z \mapsto \int_\Om f(z,t)h(t)\, \mu(\ud{t})$ is
  holomorphic and
  \[
    D^\alpha \Bigl( z \mapsto \int_\Om f(z,t)h(t)\,
    \mu(\ud{t})\Bigr) =  z\mapsto \int_\Om D^\alpha_z f(z,t) h(t)\,
    \mu(\ud{t}).
\]
\item \label{i.int.t.ifHinf.c} The measurable function
  $(t\mapsto \vphi( f_t)) : \Omega \to \C$ is
  $p$-integrable with
  \[
    \Bigl(\int_\Om \abs{\vphi(f_t)}^p\, \mu(\ud{t})\Bigr)^\frac{1}{p}
    \le \norm{\vphi} \sup_{z\in O} \norm{F(z)}_p,
  \]
  and it is contained in the subspace
  $\cls{\spann}F(O) \subseteq \Ell{p}(\Om)$.
\item \label{i.int.t.ifHinf.d}
  $\displaystyle \int_\Omega\vphi( f_t) h(t) \, \mu(\ud{t}) =
  \vphi\Bigl(z \mapsto \int_\Omega f(z,t)h(t)\, \mu(\ud{t})\Bigr)$ for
  all $h\in \Ell{q}(\Om)$.
\end{aufzi}  
\end{thm}

\begin{proof} 
 (1)\  Observe that $f$ meets the conditions
of Lemma \ref{int.l.meas} and hence that the functions
\[  t \mapsto D_z^\alpha f(a, t),\quad t\mapsto \norm{f_t}_\infty\quad
 \text{and}\quad t\mapsto
\vphi(f_t)
\]
are measurable.

\prfnoi
(2)\  Let $D\subseteq \gebiet$ be a countable dense
set. Then, since each $f_t$ is
  continuous,
  \[
    \bigcup_{z\in \gebiet} \set{F(z) \neq 0} = \bigcup_{z\in D}
    \set{F(z)\neq 0},
  \]
where
  $\set{F(z) \neq 0} = \{ t\in \Om \suchthat f(z,t) \neq 0\}$.  The
  right-hand side is a $\sigma$-finite subset of $\Om$. Hence, we may
  suppose that $\mu$ is $\sigma$-finite. Accordingly, we fix
  measurable subsets $\Omega_n \subseteq \Om$ with
  $\mu(\Om_n) < \infty$ and $\Omega_n \nearrow \Omega$.

\prfnoi
(3)\ For each $n\in \N$ let
  \[
    f_n : \gebiet \times \Om \to \C, \quad f_n(z,t) := \frac{n}{n +
      \norm{f_t}_\infty} f(z,t)\cdot \car_{\Omega_n}(t) \qquad(z\in
    \gebiet,\, t\in \Omega).
  \]
  Then $f_n$ is bounded, measurable in the second and holomorphic in
  the first variable. Applying Theorem~\ref{bdd.t.mfHinf} with
  $E = E_n := \{ g\in \BM(\Om) \suchthat g \car_{\Om_n^c} = 0\}$ we
  conclude that the function
  \[
    F_n  : \gebiet \to E_n,\qquad F_n(z) := f_n(z, \cdot)
  \]
  is bounded and holomorphic and
  \[
    D^\alpha F_n(a)(t) = D^\alpha_z f_n(a, t) = \frac{n}{n +
      \norm{f_t}_\infty} D^\alpha_z f(a,t)\cdot \car_{\Omega_n}(t)
    \qquad(t\in \Om).
  \]
  Since $\mu(\Omega_n) < \infty$, $E_n \subseteq \Ell{p}(\Om)$
  continuously and hence $F_n\in \Ha^\infty(\gebiet;
  \Ell{p}(\Om))$. Clearly, the sequence $(F_n)_n$ bp-converges to $F$,
  and hence $F \in \Ha^\infty(\gebiet; \Ell{p}(\Om))$.
 Moreover, for each $a\in \gebiet$
\[   D^\alpha F_n(a) \to  D^\alpha F(a) \quad \text{in $\Ell{p}$
  and}\quad   D^\alpha_z f_n(a, \cdot) \to D^\alpha_z f(a,\cdot) \quad
\text{pointwise}.
\] It follows  that
$D^\alpha F(a) = D^\alpha_z f(a,\cdot)$ almost everywhere and the
proof of \ref{i.int.t.ifHinf.a} is
  complete. 

\prfnoi
(4) Assertion \ref{i.int.t.ifHinf.b} follows from
  \ref{i.int.t.ifHinf.a} and Lemma \ref{hol.l.lin} on noting that
  integration against $h\in \Ell{q}(\Om)$ is a bounded linear
  functional on $\Ell{p}(\Om)$.

  \prfnoi (5)\ Since $\Omega_n$ has finite measure, for each
  $h\in \Ell{q}(\Om)$, the map
  $g \mapsto \int_{\Omega_n} g \, h \, \ud{\mu}$ defines a
  bounded linear functional on $\BM(\Om)$. Hence, by
  \ref{i.bdd.t.mfHinf.d} of Theorem~\ref{bdd.t.mfHinf} applied to
  $f_n$,
  \[
    \int_{\Omega_n} \frac{n}{n + \norm{f_t}_\infty} \vphi(f_t) h(t)\,
    \mu(\ud{t}) = \vphi\Bigl( z \mapsto \int_{\Omega_n} \frac{n}{n +
      \norm{f_t}_\infty} f(z,t)h(t)\, \mu(\ud{t})\Bigr).
  \]
  Varying $h$ we arrive at
  \[
    \int_{\Om_n} \Bigl(\frac{n}{n+\norm{f_t}_\infty}\Bigr)^p
    \abs{\vphi(f_t)}^p\, \mu(\ud{t}) \le \norm{\vphi}^p \sup_{z\in
      \gebiet}\norm{F(z)}_p^p < \infty.
  \]
  When $n\to \infty$ it follows that
  $(t\mapsto \vphi(f_t)) \in \Ell{p}(\Omega)$ and
  \[
    \norm{t\mapsto \vphi(f_t)}_p \le \norm{\vphi} \sup_z \norm{F(z)}_p.
  \]
  Moreover, by the bp-continuity of $\vphi$,
  \[
    \int_\Omega \vphi(f_t)h(t)\, \mu(\ud{t}) = \vphi\Bigl(z \mapsto
    \int_\Omega f(z,t)h(t)\, \mu(\ud{t})\Bigr),
  \]
  which is \ref{i.int.t.ifHinf.d}. The remaining part of
  \ref{i.int.t.ifHinf.c} is a consequence of \ref{i.int.t.ifHinf.d}
  and the Hahn--Banach theorem.
\end{proof}


The preceding theorem just covers the case $1\le p < \infty$. A result
for $p = \infty$ needs a special assumption on
the measure space. 
A  measure space $(\Omega, \Sigma, \mu)$ is called
{\em semi-finite} if for every $B \in \Sigma$ with $\mu(B)={+}\infty$ there
exists some $A\in \Sigma$, $A\subseteq B$ such that $0 < \mu(A) < \infty$.
It is a well-known (and easy-to-prove) fact that the unit ball of $\Ell{1}(\Omega)$ is
a norming set for $\Ell{\infty}(\Omega)$ if and only if $(\Omega,
\Sigma, \mu)$ is semi-finite. Moreover, on a semi-finite measure
space, a measurable function $f$ vanishes almost everywhere  if and only if
for each $g\in \Ell{1}(\Om)$, the product $fg$ vanishes almost everywhere.

\begin{thm}\label{int.t.Linf}
Let $(\Omega, \Sigma, \mu)$ be a semi-finite measure space,
  $\gebiet \subseteq \C^d$ an open subset of $\C^d$  and
\[  f: \gebiet \times \Omega \to \C
\]
a function with the following properties:
\begin{itemize}
\item $\forall z \in \gebiet: \;  F(z) := f(z, \cdot)$ is measurable 
  and  $\displaystyle \sup_{z\in \gebiet} \underset{t \in \Omega}{\wessup}
    \abs{f(z,t)} < \infty$; \smallskip

\item $\forall t \in \Omega: \; f_t := f(\cdot, t)$ is bounded and
  holomorphic.\smallskip
\end{itemize}
Fix $\alpha \in \N_0^d$, $a\in \gebiet$,  and a bp-continuous linear functional $\vphi:
\Ha^\infty(\gebiet)\to \C$. 
 Then the following
assertions hold:
\begin{aufzi}
\item \label{label.Linf.t.a}
  $F \in \Ha^\infty(\gebiet; \Ell{\infty}(\Omega))$
  and
  $D^\alpha F(a) = D^\alpha_z f(a,\cdot)$ $\mu$-almost everywhere.
\item \label{label.Linf.t.b} For each $g\in \Ell{1}(\Om)$ the
  function $\dps z \mapsto \int_\Om f(z,t)g(t)\, \mu(\ud{t})$ is
  holomorphic and
  \[
    D^\alpha \Bigl( z \mapsto \int_\Om f(z,t)g(t)\,
    \mu(\ud{t})\Bigr)(a) = \int_\Om D^\alpha_z f(a,t) g(t)\,
    \mu(\ud{t}).
\]
\item \label{label.Linf.t.c} The measurable function
  $t\mapsto \vphi( f_t)$ is 
  essentially bounded with
  \[
    \underset{t \in \Omega}{\wessup} \abs{\vphi(f_t)}  \le \norm{\vphi} \sup_{z\in O} \bignorm{F(z)}_{\Ell{\infty}(\Omega)},
  \]
  and it is contained in the subspace
  $\cls{\spann\, F(O)}^{\sigma(\Ell{\infty}, \Ell{1})} \subseteq \Ell{\infty}(\Om)$.
\item \label{label.Linf.t.d}
  $\displaystyle \int_\Omega\vphi( f_t) g(t) \, \mu(\ud{t}) =
  \vphi\Bigl(z \mapsto \int_\Omega f(z,t)g(t)\, \mu(\ud{t})\Bigr)$ for
  all $g\in \Ell{1}(\Om)$.
\end{aufzi}  
\end{thm}

\begin{proof}
  We start by noticing as in the proof of Theorem~\ref{int.t.ifHinf},
  that by Lemma \ref{int.l.meas} the functions
  $t\mapsto D_z^\alpha f(a, t), \vphi(f_t)$ are measurable.

  For any $g \in \Ell{1}(\Omega)$, the function
  $(z,t) \mapsto f(z, t)g(t)$ satisfies the assumptions of
  Theorem~\ref{int.t.ifHinf}, for the case $p{=}1$.  Hence
  \ref{i.int.t.ifHinf.b} of that theorem (with $h= \car$) is just the
  same as \ref{label.Linf.t.b} of the present theorem. Recall that,
  since $(\Om, \Sigma,\mu)$ is semi-finite, the unit ball of
  $\Ell{1}(\Omega)$ is norming for $\Ell{\infty}(\Omega)$. Since
  $F: \gebiet \to \Ell{\infty}(\Omega)$ is bounded, it follows from
  \ref{label.Linf.t.b} and Theorem~\ref{hol.t.char} that
  $F \in \Ha^\infty(\gebiet; \Ell{\infty}(\Omega))$.

  From part \ref{i.int.t.ifHinf.b} of Theorem~\ref{int.t.ifHinf} it
  follows that $D^\alpha F(a) g = D_z^\alpha f(a,\cdot) g$
  $\mu$-almost everywhere. As $g\in \Ell{1}(\Om)$ is arbitrary here
  and $D_z^\alpha f(a,\cdot)$ is measurable, this implies that
  $D^\alpha F(a) = D_z^\alpha f(a,\cdot)$ $\mu$-almost
  everywhere. Hence, \ref{label.Linf.t.a} is proved.

\prfnoi
  Part \ref{i.int.t.ifHinf.c} of Theorem~\ref{int.t.ifHinf} implies
  that $t \mapsto \vphi(f_t)g(t)$ is integrable with
  \[
    \int_\Om \abs{\vphi(f_t) g(t)}\, \mu(\ud{t}) \le \norm{\vphi}
    \sup_z \norm{F(z)g}_{\Ell{1}} \le \norm{\vphi} \sup_z
    \norm{F(z)}_{\Ell{\infty}} \cdot \norm{g}_{\Ell{1}}.
\]
Varying $g$ yields that $t\mapsto \vphi(f_t)$ is essentially bounded
with
\[
  \underset{t \in \Omega}{\wessup} \abs{\vphi(f_t)} \le \norm{\vphi}
  \sup_z \norm{F(z)}_{\Ell{\infty}},
\]
which is part \ref{label.Linf.t.c} halfway. 
Again by Theorem~\ref{int.t.ifHinf}~\ref{i.int.t.ifHinf.d}
applied with $h = \car$, yields \ref{label.Linf.t.d}. Finally,  it
follows from \ref{label.Linf.t.d} and a standard application of the  Hahn--Banach
theorem  
that the function $t \mapsto \vphi(f_t)$ (as an element of
$\Ell{\infty}(\Om)$) is contained in the $\sigma(\Ell{\infty},
\Ell{1})$-closure of $\spann F(\gebiet)$. This
concludes the proof. 
\end{proof}


\begin{remarks}\label{int.r.Stein}
  1)\ For $d{=}1$ and $p{=}1$, some parts of
  Theorem~\ref{int.t.ifHinf} have been proved by Mattner in
  \cite{Mattner2001}.  Mattner also has shown in
  \cite[Counterexample~1]{Mattner2001} that if one replaces the first
  condition in Theorem~\ref{int.t.ifHinf} by the weaker one
\[  \int_\Om \abs{f(z,t)}\, \mu(\ud{t}) < \infty \quad \text{for all
  $z\in \gebiet$},
\]
then the function $z \mapsto \int_\Om f(z,t)\, \mu(\ud{t})$ need not
be continuous, let alone holomorphic. In particular, assertion
\ref{i.int.t.ifHinf.a} may fail.

In Section~\ref{s.matt} below, we shall review Mattner's results from
\cite{Mattner2001} and relate them to ours.

\medskip \noindent 2)\ The following result by Stein tells us that
{\em each} holomorphic $\Ell{p}$-valued function arises in the way
considered in Theorem~\ref{int.t.ifHinf}.

\smallskip 
\noindent
{\bf Theorem (Stein).}\
{\em 
  Let $X$ be a complex Banach space, $(\Omega, \Sigma, \mu)$ a $\sigma$-finite
  measure space, $1\le p < \infty$, and $\gebiet \subseteq \C$ an open
  set. Let $F: \gebiet \to \Ell{p}(\Omega; X)$ be a holomorphic function. Then
  there exists a function $f: \gebiet \times \Omega \to X$ such that
  \begin{itemize}
    \item $f$ is strongly (product) measurable;
    \item for every $t \in \Omega$, $f(\cdot, t)$ is holomorphic;
    \item for every $z \in \gebiet$, $f(z, \cdot) = F(z)$ almost everywhere. 
  \end{itemize}
}

\noindent This theorem goes back to the lemma on page~72 of Stein's
book \cite{Stein:LPT} for the special case of orbits of holomorphic
semigroups on sectors.  Desch and Homan in \cite{DeschHoman} proved
the theorem in full generality and with all the details.
\end{remarks}

\section{Mattner's results}\label{s.matt}

Theorem~\ref{int.t.ifHinf} and in particular the Fubini-type result in
assertion \ref{i.int.t.ifHinf.d} may be surprising on first glance,
since the joint measurability of the function $f$ is not assumed (and
not needed in the proof). However, as the following result shows,
joint measurability is automatic, at least in the case we consider
here ($\gebiet \subseteq \C^d$).

\begin{lemma}[Mattner {\cite[p.33]{Mattner2001}}]\label{matt.l.meas}
  Let $\gebiet$ be a second countable topological space, let
  $(\Om,\Sigma)$ be a measurable space and let $X$ be a metric
  space. Furthermore, let
  \[
    f: \gebiet \times \Om \to X
  \]
  be a mapping such that $f(\cdot, t): \gebiet \to X$ is continuous
  for each $t\in \Om$ and $f(z,\cdot): \Om \to X$ is measurable
  $\Sigma$-to-$\Borel$. Then $f$ is measurable
  $(\Borel\tensor\Sigma)$-to-$\Borel$.
\end{lemma}

\begin{remark}
  Joint measurability results appear to have a long tradition, see
  e.g. \cite[p.42]{Grothendieck1953}, \cite{Averna1986},
  \cite[Lemma~4.51]{AliprantisBorderIDA}. However, they are seldom
  mentioned in courses on measure theory and not widely
  known. Mattner's version appears to be the strongest so far where
  measurability in one variable is paired with continuity in the
  other.
\end{remark}

\noindent  With the help of Lemma~\ref{matt.l.meas}, Mattner in
\cite{Mattner2001} established the following theorem (slightly adapted
notationally).

\begin{thm}[Mattner {\cite[p.32]{Mattner2001}}]\label{matt.t.matt}
  Let $(\Om, \Sigma,\mu)$ be a measure space, let
  $\emptyset \not= O \subseteq \C$ be open, and let
  $f: O\times \Om\to \C$ be a function subject to the following
  assumptions:
\begin{mattnerlist}{A}
\item \label{i.matt.t.matt.a1} $f(z, \cdot)$ is $\Sigma$-measurable for every $z \in O$, 
\item \label{i.matt.t.matt.a2} $f (\cdot ,t)$ is holomorphic for every $t \in \Om$,
\item \label{i.matt.t.matt.a3} $\displaystyle  \int \abs{f (\cdot, t)} \mu(\ud{t})$ is locally bounded.
\end{mattnerlist}
Then $\displaystyle z\mapsto \int_\Om f(z,t)\, \mu(\ud{t})$ is
holomorphic and may be differentiated under the integral. More
precisely, for each $n\in \N_0$:
\begin{mattnerlist}{C}
\item \label{i.matt.t.matt.c1} $D^n_zf$ is
  $\Borel(O) \tensor \Sigma$-measurable and, for every
  $\emptyset \not=A \subseteq O$, the function
  \( t \mapsto \sup_{z\in A} \abs{D^n_zf(z, t)}\) is
  $\Sigma$-measurable,
\item \label{i.matt.t.matt.c2} If $\displaystyle K \subseteq O$ is
  compact, then
  \( \displaystyle \sup_{z\in K} \int_\Om \abs{D^n_z f(z,t)}\,
  \mu(\ud{t}) < \infty. \)
\item \label{i.matt.t.matt.c3}
  $\displaystyle z \mapsto \int_\Om f (z, t)\, \mu(\ud{t})$ is
  holomorphic on O with
  \[
    D^n_z \int_\Om f (z, t)\, \mu(\ud{t}) = \int_\Om D^n_z f (z, t)\,
    \mu(\ud{t}).
\]
\end{mattnerlist}
\end{thm}

\noindent Assertion~\ref{i.matt.t.matt.c3} in Theorem~\ref{matt.t.matt} is
covered (literally) by part \ref{i.int.t.ifHinf.b} of Theorem
\ref{int.t.ifHinf}. The joint measurability assertion in
C\ref{i.matt.t.matt.c1} follows directly from \ref{i.int.t.ifHinf.a}
and Lemma~\ref{matt.l.meas}; and the remaining part of
C\ref{i.matt.t.matt.c1} follows since the supremum is effectively a
supremum over a countable subset of $A$. (Mattner employs the same
argument).  Finally, C\ref{i.matt.t.matt.c2} is a straightforward
consequence of \ref{i.int.t.ifHinf.a} and the following general
result.

\begin{lemma}\label{matt.l.order-bounded}
  Let $\gebiet \subseteq \C^d$ open, $E$ a  Banach lattice
  and $F: \Om \to E$ holomorphic. Then for each
  compact $K \subseteq \gebiet$ the set $F(K)$ is order bounded in
  $E$, i.e., there is $0 \le u \in E$ such that $\abs{F(z)} \le u$
  for all $z\in K$.
\end{lemma}

\begin{proof}
  We only treat the case $d=1$, the general case being analogous (but
  more technical to write down). By compactness of $K$ it suffices to
  prove that each point $a\in \gebiet$ has a neighborhood $U_a$ such
  that $F(U_a)$ is order-bounded. Without loss of generality, $a=0$.
  Write $F$ as a convergent power series
  \[
    F(z) = \sum_{n=0}^\infty u_n z^n \qquad (\abs{z} < r)
  \]
  for some $r > 0$ and $u_n \in E$. Making $r$ smaller if necessary,
  we have
  \[
    \sum_{n=0}^\infty \norm{u_n} r^n < \infty.
  \]
  The series $u := \sum_{n=0}^\infty \abs{u_n} r^n$ is
  (norm-absolutely) convergent and we obtain for $\abs{z}< r$ that
  \( \abs{F(z)} \le u \), as claimed.
\end{proof}

\noindent Alternatively, one may prove Lemma~\ref{matt.l.order-bounded} by means
of the Cauchy formula, and this is exactly what Mattner does in
\cite{Mattner2001} for $E = \Ell{1}(\Om)$.

\begin{remark}
By adapting the proof, Lemma~\ref{matt.l.order-bounded} can be
generalized from Banach lattices to ordered Banach
spaces with generating positive cone, but we refrain from doing so
here.  The result may be known, but we do not know of
a direct reference, hence we have included the simple proof. Basically, the
argument is present in the proof of \cite[Lemma on p.72]{Stein:LPT},
already mentioned in Remark~\ref{int.r.Stein}, 2). 
\end{remark}

\noindent In the remainder of this section, we show how 
Theorem~\ref{int.t.ifHinf}  can be derived from Theorem
\ref{matt.t.matt}, as long as one supposes in addition
that $\vphi$ is integration against a finite measure. Since the latter is
no restriction, by Remark~\ref{rem:cooper}, this  yields an
alternative proof of Theorem~\ref{int.t.ifHinf} for the case
$p{=}1$. (With a little more effort, one would also obtain a
proof for the case $p \ge 1$.) Although Mattner only treated the
single-variable case, we expect his approach to work also for
functions of several variables.

\begin{proof}[Sketch of an alternate proof of
  Theorem~\ref{int.t.ifHinf} building on \cite{Mattner2001}]
  As already said, our point of departure is as follows: $p=1$, $f$ is
  already known to be joint measurable (Lemma~\ref{matt.l.meas}) and
  $\mu$ is $\sigma$-finite (by the same argument as (1) in our
  original proof). Moreover, \ref{i.int.t.ifHinf.b} holds, i.e., the
  function $z \mapsto \int_\Om F(z)h\,\ud{\mu}$ is holomorphic for all
  $h\in \Ell{\infty}(\Om)$. Finally, $\vphi$ is integration with
  respect to a complex measure $\nu$.

  \prfnoi From \ref{i.int.t.ifHinf.b} and the boundedness of $F$ it follows
  from Theorem~\ref{hol.t.char} that $F$ is holomorphic. Then from
  \ref{i.int.t.ifHinf.b} and Lemma~\ref{hol.l.lin} it follows that
  $D^\alpha F$ is represented by $D^\alpha_zf(z, \cdot)$, and hence
  \ref{i.int.t.ifHinf.a} is proved.

  \prfnoi Next, by Fubini-Tonelli, $f$ is $\abs{\nu}\times \mu$-integrable and
  we can interchange the order of integration. This yields the first
  part of \ref{i.int.t.ifHinf.c} (measurability and integrability of
  $t\mapsto \vphi(f_t)$ as well as the norm estimate) and
  \ref{i.int.t.ifHinf.d}.

  \prfnoi The remaining part of \ref{i.int.t.ifHinf.c} is now proved
  exactly as in the original proof.
\end{proof}

\bigskip
\noindent
{\bf Acknowledgements:} This paper was finalized during a research
stay of the second author at the Dipartimento di Matematica ``Guiseppe
Peano'' of the Universitá di Torino, Italy. Markus Haase is grateful
to this institution for its hospitality and, in particular, to Jörg
Seiler for his kind invitation. 

Both authors thank the anonymous referee for his valuable remarks, in
particular for his references to the works of Grosse-Erdmann and
Mujica, which led to a considerable improvement.

\def\PREPRINT{preprint}
\def\PREPARATION{in preparation}

\def\cprime{$'$}
\providecommand{\bysame}{\leavevmode\hbox to3em{\hrulefill}\thinspace}

\end{document}